\documentclass[12pt]{amsart}
\usepackage[top=2cm, bottom=5cm, left=3cm, right=1cm]{geometry}

\usepackage{latexsym}
\usepackage{amssymb}

\usepackage{amsthm}
\usepackage{latexsym}
\usepackage{longtable}
\usepackage{epsfig}
\usepackage{amsmath}
\usepackage{hhline}
\usepackage{color}

\usepackage{listings}

\newtheorem{theorem}{Theorem}

\newtheorem{lemma}{Lemma}
\newtheorem{corollary}{Corollary}
\newtheorem{proposition}{Proposition}

\DeclareMathOperator{\ann}{ann}

\DeclareMathOperator{\charak}{char}
\DeclareMathOperator{\Der}{Der}
\DeclareMathOperator{\LDer}{LDer}

\DeclareMathOperator{\ZDer}{ZDer}
\DeclareMathOperator{\IDer}{IDer}

\DeclareMathOperator{\card}{card}
\DeclareMathOperator{\supp}{supp}

\begin{document}

\title{Derivations of group rings}

\author[Artemovych, Bovdi, Salim]{Orest D.~Artemovych, Victor A.~Bovdi, Mohamed A.~Salim}
\footnote{Corresponding Author: \quad V.A.~Bovdi\\
The research  was supported by  the UAEU  UPAR  grant  G00002160.}
\address{Department of Applied   Mathematics,  Cracow University of Technology, Cracow,  Poland}
\email{artemo@usk.pk.edu.pl}

\address{UAEU, United Arab Emirates}
\email{vbovdi@gmail.com; msalim@uaeu.ac.ae }

\keywords{Group ring,  derivation, locally finite group, solder, torsion-free group, nilpotent group,
differentially trivial ring, nilpotent Lie ring, solvable Lie ring. }
\subjclass{20C05, 16S34,   20F45, 20F19, 16W25}

\maketitle

\noindent
\begin{abstract}
Let $R[G]$ be the group ring of a group $G$ over an associative ring $R$ with unity such that all prime divisors of orders of elements of $G$ are invertible in $R$. If $R$ is finite and $G$ is a Chernikov (torsion $FC$-) group, then each $R$-derivation of $R[G]$ is inner. Similar results also are obtained for other  classes of groups $G$ and rings $R$.
\end{abstract}

\markboth{Artemovych, Bovdi, Salim}{Derivations of group rings}

\centerline{\it Dedicated to the memory of Professor V.I.~Sushchansky}

\section{ Introduction}

Let $B$ be  an associative ring not necessarily   with unity $1$.
 An additive  map   $\delta : B\rightarrow B$ is called {\it a derivation} of
$B$ if
\begin{equation} \label{E:1}
\delta (ab)=\delta (a)b+a\delta (b)\qquad\qquad  (a, b\in B).
\end{equation}
The set $\Der B$ of all derivations of $B$ is a Lie ring with respect to a point-wise addition and a point-wise Lie multiplication of derivations. The zero map  $0: B\ni r\mapsto 0\in B$ is a derivation of $B$. A ring $B$ having only zero derivation (i.e., $\Der B=0$) is called {\it differentially trivial} (see \cite{Artemovych}). The rule
\[
\partial_a:B\ni x\mapsto ax-xa\in B
\]
determines a derivation $\partial_a$ of $B$ (so-called {\it an inner derivation of $B$} induced by $a\in B$). If $\delta\in \Der B$ is not inner, then it is called {\it outer}.
The set $\IDer B:=\{\partial_a\mid a\in B\}$ of all inner derivations of  $B$ is an ideal of  the derivation ring $\Der B$ (see \cite{Jacobson}).

Throughout the paper  $p$ is a prime, ${\mathbb Z}$  the ring of
integers  and ${\mathbb N}$ the set of positive integers.  If $X,Y\subseteq B$, then a subgroup of the additive group $B^+$ generated by the Lie commutators
$[g,h]= gh-hg$, where  $g\in X$ and  $h\in Y$, is denoted by $[X,Y]$. The commutator ideal $C(B)$ is an ideal of  $B$ generated by the set $[B,B]$. Furthermore,
$\ann K=\{ r\in B\mid rK=Kr=0\}$ is the annihilator of $K$ in $B$,
$F(B)=\{ a\in B\mid a\ \text{is of finite index in}\ B^+\}$ is the torsion part of $B$, ${\mathbb P}(B)$ is the prime radical of $B$, i.e., the intersection of all prime ideals of $B$, and
$Q(B)$ is the classical quotient ring of  $B$ (see e.g. \cite[Chapters~2~and~5]{McConnell_Robson}).

In the next  $R$ is an associative ring with  unity $1$, $G$ a multiplicative group and
$R^{\mathbb Z}$ is a direct sum of $\card {\mathbb Z}$ copies of $R$. The torsion part of a group  $G$ we denote by $\tau G$ and the set of primes that divide orders of elements of  $\tau G$ by $\pi (G)$, respectively. Let $\Delta^+(G)$ be the set of
all elements $g\in G$ of finite orders such that
$|G:C_G(x)|<\infty$.  The center of a group (respectively a ring) $X$ we denote by $Z(X)$.

Let $R[G]$ be the group algebra of a group $G$ over a ring $R$.
The group of units $U(R[G])$ of the group algebra  $R[G]$ can be written as
$U(R[G])=U(R)\times V(R[G])$, where $U(R)$ is the group of units of the ring $R$ and
\[
V(R[G]):=\Big\{\; \sum_{g\in G}\alpha_gg\in U(R[G])\quad \mid\quad  \sum_{g\in
G}\alpha_g=1\; \Big\}
\]
is the subgroup of the normalized units of $R[G]$.
It follows that
\begin{equation} \label{EE:5}
[U(R),V(R[G])]=0.
\end{equation}
Finally, any unexplained terminology is standard as in  \cite{Bovdi_book, Fuchs_book, Herstein_book, Robinson_book}.

If $\delta(R)=0$, then a derivation $\delta$ of a group ring $R[G]$ is called {\it an $R$-derivation} of $R[G]$. The set $\Der_RR[G]$ of all $R$-derivations of $R[G]$ is a subring of the Lie ring $\Der R[G]$.

Smith (see \cite{Smith}) was first who started to study the derivations in group rings. For instance, she showed that in  group  rings of torsion-free nilpotent groups always there exists an outer derivation. In certain papers (see \cite{Ferrero_Giambruno_Milies, Smith,  Spiegel})  the properties of group rings $R[G]$ such that its every derivation is inner  were studied. In \cite{Ferrero_Giambruno_Milies} it  was  proved that  $R$-derivations
of a group ring $R[G]$ of a center-by-finite group $G$ over a
semiprime ring $R$, where $\charak R=0$ or every prime $\charak R
\notin \pi (G)$, are inner.  In \cite{Burkov} Burkov  proved that every $K$-derivation $\theta$ of a group ring $K[G]$, where $K$ is a commutative ring, $G$ is a torsion group
and the order of each element of $G$ is invertible in $R$, is
generalized inner (i.e., there exists $a\in (K[G])^*$ such that
$\theta (x)=xa-ax$ for each $x\in K[G]$, where $(K[G])^*$ is a
$(K[G],K[G])$-bimodule of all functions from $G$ in $K$). Derivations of certain rings were investigated  in \cite{Artemovych_3, Artemovych_4, Artemovych_2, Baclawski, Bai_Meng_He, Bavula_1,  Bavula_2,  Kaygorodov_Popov_1, Kaygorodov_Popov_2,  Kharchenko_1, Kharchenko_2, Kiss_Laczkovich, Lanski, Lanski_Montgomery, Penkava_Vanhaecke, Spiegel, Tang_Liu_Xu}.

Our main results are  the following

\begin{theorem} \label{T:1}
Let  $R$ be  a  finite ring. If  $G$ is a Chernikov group such that $\pi (G)\cap \pi (R^+)=\varnothing$, then every $R$-derivation of $R[G]$ is inner.
\end{theorem}

\begin{theorem} \label{T:2}
Let $R$ be a  ring. If $G$ is a torsion abelian group such that all primes $p\in \pi (G)$ are invertible in $R$, then $\Der_RR[G]=0$.
\end{theorem}

\begin{theorem}\label{T:3}
 Let $R$ be a finite ring, $G$ a countable
locally finite group such that every prime $p\in \pi (G)$ is
invertible in $R$. If $G$ has a finite subgroup $H$ with the
finite centralizer $C_{R[G]}(H)$, then every $R$-derivation of
$R[G]$ is inner.
\end{theorem}

Also if $R$ is finite and $G$  is a torsion $FC$-group, then every $R$-derivation
of $R[G]$ is inner (see Corollary \ref{C:20}).

Recall that a  ring $B$ is {\it prime} if, for any ideals $X,Y$ of $B$, we have
$XY=0$ implies that    $X=0$ or $Y=0$. A ring $B$ without nonzero
nilpotent ideals is called {\it  semiprime}.

In \cite{Kaplansky}  Kaplansky has conjectured that the group algebra
${\mathbb F}[G]$ of a torsion-free group $G$ over a field $\mathbb
F$ has no nontrivial idempotents. We prove the following.

\begin{theorem} \label{T:4}
If $R$ is a prime ring of characteristic $\neq 2,3$ with the non-central unit group
$U(R)$ and $G$ a group such that $\Delta^+(G)=1$, then the following conditions hold:
\begin{itemize}
\item[$(i)$] if $G$  is non-abelian, then $R[G]$  has only trivial
idempotents;
\item[$(ii)$] if  every prime $p\in \pi (G)$
is invertible in $R$   and
$R[G]$ has a non-trivial idempotent, then $G$ is abelian,  $\delta (\tau G)=0$ for each
$\delta \in \Der R[G]$ (and so   $\IDer_RR[G]=0$).

\end{itemize}
\end{theorem}

A Lie ring $D$ is called {\it nilpotent} if $\gamma_{n+1}D=0$ for
some $n\in\mathbb{N}$, where
\[
\gamma_1D:=D, \ldots, \gamma_{k+1}D:=[\gamma_kD,D],\ldots  \qquad (k\in\mathbb{N}).
\]
 The
least such $n$ is the {\it class of nilpotency} of $D$. If $n\in\mathbb{N}$ and
$x_1,\ldots, x_n, x, y\in B$, then the left normed commutators are
defined inductively:
\begin{itemize}
\item[$\bullet$] $[x,y]:=[x,_1 y]:=xy-yx$;
\item[$\bullet$] $[x,_{n+1}y]:=[[x,_n y],y]$;
\item[$\bullet$] $[x_1,\ldots ,x_{n-1},x_n]:=[[x_1,\ldots ,x_{n-1}],x_n]$.
\end{itemize}
The lower central chain of $B$ is
\[
B=\gamma_1B\supseteq \gamma_2B\supseteq \cdots \supseteq \gamma_nB\supseteq \cdots ,
\]
where $\gamma_{n+1}B:=[\gamma_nB,B]$ for integers $n\geq 1$. If $\gamma_{n+1}B=0$ and
$\gamma_nB\neq 0$ for some  $n\in\mathbb{N}$, then $B$ is called {\it Lie nilpotent
of class} $n$. Each associative ring $B$ can be considered as a Lie ring $B^L$ under operations ``$[-,-]$"{} and ``$+$"{}, which is called the {\it associated Lie ring} of $B$. It is easy to see that an associative ring $B$ is Lie nilpotent if and only if  the associated Lie ring $B^L$ is nilpotent.

A Lie ring $D$ is called {\it solvable} if $D^{(n+1)}=0$ for some $n\in {\mathbb N}$, where
$D^{(1)}:=D$, $D^{(n+1)}:=[D^{(n)},D^{(n)}]$. The least such $n$ is
the {\it derived length} of $D$.
The derived chain of $B$ is given by
\[
B=B^{(0)}\supseteq B^{(1)}\supseteq \cdots \supseteq
B^{(n)}\supseteq \cdots ,
\]
where $B^{(n)}:=[B^{(n-1)},B^{(n-1)}]$,
$n\geq 1$. An associative ring $B$ is {\it Lie solvable of length
$n$} if $B^{(n)}=0$, $n$ least. An associative ring $B$ is Lie
solvable if and only if  the associated Lie ring $B^L$ is
solvable.

\begin{theorem} \label{T:5}
Let $B$ be a semiprime ring. The Lie ring
$\Der B$ is  nilpotent (respectively solvable)  if and
only if $\Der B=0$, i.e., $B$ is differentially trivial.
\end{theorem}

As an immediate consequence of the previous result  we obtain the following.

\begin{theorem} \label{T:6}
Let $\mathbb F$ be a division ring of characteristic $p\geq 0$, $G$ a group. If  each  prime $q\in \pi
(G)$ is invertible in $\mathbb F$, then $\Der {\mathbb F}[G]$ is a
nilpotent (respectively solvable) Lie ring if and only if $\Der
{\mathbb F}[G]=0$.
\end{theorem}

The set of all sequences $\{ D_n\in \Der R\}_{ n\in \mathbb{Z}}$ such that   $D_n(a)=0$ for   every $a\in R$ and almost all $n\in \mathbb{Z}$
we denote by ${\mathcal LF}(R)$. Evidently,  ${\mathcal LF}(R)$ is a Lie ring with respect to an addition ``$+$"{} and the  Lie multiplication ``$[-,-]$"{} defined by the rules:
\[
\begin{split}
{\{ D_n\}_{ n\in \mathbb{Z}} }+{\{ K_n\}_{ n\in \mathbb{Z}} }={\{ D_n+K_n\}_{ n\in \mathbb{Z}} };\qquad
[{\{ D_n\}_{ n\in \mathbb{Z}} },{\{ K_n\}_{ n\in \mathbb{Z}} }]={\{ [D_n,K_n]\}_{ n\in \mathbb{Z}} },\\
\end{split}
\]
where
${\{ D_n\}_{ n\in \mathbb{Z}}  },{\{ K_n\}_{ n\in \mathbb{Z}}  }\in {\mathcal LF}(R)$.  As  in    \cite{Baclawski, Burkov_2},  the  family ${\mathcal LF}(R)$  is called {\it locally finite}.  The next  result  is analogously with \cite[Theorem 2]{Baclawski}.

\begin{theorem} \label{T:7}
Let $R$ be a  ring.  If $G$ is a countable  torsion-free abelian group,  then
 the following conditions hold:
 \begin{itemize}
 \item[$(i)$]  $\Der R[G]\cong {\mathcal LF}(R)\oplus (Z(R[G]))^{\mathbb Z}$ as  Lie rings;
 \item[$(ii)$] each nonzero $R$-derivation of  $R[G]$ is outer.
 \end{itemize}
\end{theorem}

In section  4  we prove that the behavior of the  $R$-derivations of a group ring  $R[G]$ of  a nilpotent group   $G$ with some restriction is very similar to the nilpotency property of  derivations (see Proposition \ref{P:28}).

\vskip20pt

\section{Solders and central derivations}

If $H\leq G$, then $\mathfrak{I}_R(H)$ is a right ideal of $R[G]$ generated by the set $\{ 1-h\mid h\in H\}$.
If $H$ is normal in $G$, then $\mathfrak{I}_R(H)$ is an ideal of $R[G]$ and we have a ring isomorphism
\[
R[G]/\mathfrak{I}_R(H)\cong R[G/H]
\]
 \cite[Proposition 1, p. 17]{Bovdi_book}. Therefore, we have some Lie ring isomorphism
\[
\Der (R[G]/\mathfrak{I}_R(H))=\Der R[G/H].
\]

A ring $R$ is called {\it $n$-torsion-free} ($n\in \mathbb{N}$) if the implication $nx=0\Rightarrow x=0$ is true
for each  $x\in R$.

\begin{lemma} \label{L:8}
 Let $R$ be a ring. The  following conditions hold:
\begin{itemize}
\item[$(i)$] if $U$ is a subsemigroup of the multiplicative semigroup $(R,\cdot )$ of $R$ and $h: U\to R^+$ is a semigroup homomorphism, then $\delta_h: U\ni u\mapsto uh(u)\in R^+$  satisfies the Leibniz rule {\rm (see Eq. $(\ref{E:1})$)} if and only if $u[R,h(u)]=0$ for each $u\in U$;
\item[$(ii)$] if $h:(R,\cdot )\to R^+$ is a semigroup homomorphism, then we have:
\begin{itemize}
\item[$(a)$] $h(1)=0=h(0)$ and $2h(-1)=0$;
\item[$(b)$] if $R$ is $2$-torsion-free, then $h(a)=h(-a)$ for each $a\in R$;
\end{itemize}
\item[$(iii)$] if $U$ is a subgroup of the unit group $U(R)$ and $\delta \in \Der R$, then
    \[
    L_\delta :U\ni a\mapsto a^{-1}\delta (a)\in R^+
    \]
    is a group homomorphism if and only if $[a^{-1}\delta (a),U]=0$ for each  $a\in U$.
\end{itemize}
\end{lemma}

\begin{proof} Let  $a,b \in U$. \quad  $(i)$ It is easy to see  that
\[
ah(a)b+abh(b)=\delta_h(a)b+a\delta_h(b)=\delta_h(ab)=abh(ab)=abh(a)+abh(b)
\]
 if and only if $a(h(a)b-bh(a))=0$, so  the proof is done.

$(ii)$ Let   $h:(R,\cdot )\to R^+$ be  a semigroup homomorphism.

$(a)$ Clearly,  $h(1)=h(1^2)=h(1)+h(1)$  and   $h(0)=h(0^2)=h(0)+h(0)$,  so $h(1)=0=h(0)$. Furthermore,
$0=h(1)=h((-1)\cdot (-1))=h(-1)+h(-1)$.

$(b)$ Since $h(-1)=0$, we get  that
$h(-a)=h((-1)a)=h(-1)+h(a)=h(a)$ for each $a\in R$.

$(iii)$ It is easy to check that
\[
\begin{split}
 L_\delta (a)+L_\delta (b)&=L_\delta (ab)=(ab)^{-1}\delta (ab)=b^{-1}a^{-1}a\delta (b)+b^{-1}a^{-1}\delta (a)b\\
 &=b^{-1}\delta (b)+b^{-1}ba^{-1}\delta (a)+b^{-1}[a^{-1}\delta (a),b]\\
 &=L_\delta (a)+L_\delta (b)+b^{-1}[a^{-1}\delta (a),b]
\end{split}
\]
 if and only if $b^{-1}[a^{-1}\delta (a),b]=0$. Thus  $[a^{-1}\delta
(a),U]=0$.
\end{proof}

A  map $h: R\to R$ is called {\it a solder} of $R$ (see Nowicki \cite{Nowicki}, p.~46) if:
\begin{itemize}
\item[$(i)$] $(a+b)h(a+b)=a h(a)+b h(b)$ \qquad ($a,b\in R$);
\item[$(ii)$] $h(ab)=h(a)+h(b)$ \qquad ($a,b\in R\setminus \{0\}$).
\end{itemize}

\begin{lemma} \label{L:9}
Let  $h$ be a  solder of a ring $R$. The following conditions  hold:
\begin{itemize}
\item[$(i)$] $h(xy)=h(yx)$ for each  $x,y\in R$;
\item[$(ii)$] the rule $\delta_h:R\ni x\mapsto xh(x)\in R$ determines a derivation $\delta_h \in \Der R$ if and only if $x[h(x),R]=0$ for  each $x\in R$;
\item[$(iii)$] if $R$ is a prime ring, then $\delta_h \in \Der R$ if and only if $h(R)\subseteq Z(R)$ {\rm(such solder is called {\it central})};
\item[$(iv)$] if $R$ is $2$-torsion-free, then $h(2)=0$;
\item[$(v)$] if $e^2=e\in R$, then $h(e)=0$;
\item[$(vi)$] if $xy=1$ for some $x,y\in R$, then $h(x)=-h(y)$; in particular, if $x^2=1$, then $2h(x)=0$.
\end{itemize}
\end{lemma}

\begin{proof} $(i)$ Since $h(0)=0$ and $\delta_h(xy)=xyh(x)+xyh(y)$ for  each $x,y\in R$, the result follows.

 $(ii)$ It follows from Lemma \ref{L:8}$(i)$ because $\delta_h$ is additive.

$(iii)$ If  $x,y\in R$, then $\delta_h$ is a derivation $R$ if and only if
\[
xy[h(x),R]=xy[h(x),R]+xy[h(y),R]=xy[h(xy),R]=0
\]
 by the part $(ii)$.  Thus  $xR[h(x),R]=0$ and so  $[h(x),R]=0$ by the  primeness of $R$.

$(iv)$ Obviously,   $0=1\cdot h(1)+1\cdot
h(1)=(1+1)h(1+1)=2h(2)$,  therefore $h(2)=0$.

$(v)$ In as much as $h(e)=h(e^2)=h(e)+h(e)$, we have that $h(e)=0$.

$(vi)$ Evidently,  $0=h(1)=h(xy)=h(x)+h(y)$,  so $h(x)=-h(y)$. \end{proof}

The  subgroup of a group $G$ generated by the set $\{ g^n\mid g\in G\}$ we denote  by $G^n$.
The exponent of a torsion group $G$ is the following number $\exp(G)=\min\{n\in \mathbb{N}\mid x^n=1, \forall x\in G\}$.

\begin{proposition} \label{P:10}
Let $R$ be a ring,  $G$ a group and   $n\in \mathbb{N}$. The
following conditions hold:
\begin{itemize}
\item[$(i)$] $\ZDer R:=\{ \delta \in \Der R \mid \delta
(R)\subseteq Z(R)\}$ is an ideal of the Lie ring $\Der R$;
\item[$(ii)$] if $\delta \in \ZDer R$, then $\delta ([R,R])=0$;
\item[$(iii)$] if $R$ is $n$-torsion-free and the exponent $\exp (G)=n$, then $\delta (G)=0$ for  each $\delta \in \ZDer  R[G]$;
\item[$(iv)$] if  $nR=0$  and  $G$ is an
 abelian torsion-free group, then $\delta (G^n)=0$ for  each $\delta \in \ZDer R[G]$.
\end{itemize}
\end{proposition}
\begin{proof} Let $\delta \in \ZDer R[G]$ and  $d\in \Der R$.

$(i)$ Indeed, $[\delta ,d](r)=\delta (d(r))-d(\delta (r))\in Z(R)$
for any $r\in R$ and the result holds.

$(ii)$ It follows in view of the part $(i)$.

$(iii)$ The map $L_\delta :{\langle g\rangle}\to R[G]^+$ is a group homomorphism for any $\delta \in \ZDer R[G]$   and $g\in G$ by Lemma \ref{L:8}$(iii)$. Using facts that  $g^n=1$ for $g\in G$ and
\begin{equation} \label{E:2}
g^{-n}\delta (g^n)=L_\delta (g^n)=nL_\delta (g),
\end{equation}
we deduce that $L_\delta (g)=0$, so  $\delta (g)=0$.

$(iv)$  In as much as $\delta (G)\subseteq \ZDer R[G]$, the map $L_\delta: G\to R^+$ is a group homomorphism by Lemma
\ref{L:8}$(iii)$.  Hence   Eqs.   $(\ref{E:2})$ hold for any $g\in G$
and  $n\in\mathbb{N}$. Since $nL_\delta (g)=0$, we deduce that
$\delta (g^n)=0$. \end{proof}

A  derivation $\delta \in
\Der R$ is called {\it central} if $\delta (R)\subseteq Z(R)$ (see  \cite{Togo}). The
set  of all central $R$-derivations of $R[G]$ we denote by
$\ZDer_R R[G]$. It is easy to check that
\[
\ZDer_RR[G]=(\ZDer R[G])\cap (\Der_RR[G]).
\]
Clearly,  $Z(G)\leq Z(R[G])$, $Z(R)[Z(G)]\subseteq Z(Z(R)[G])\subseteq Z(R[G])$ and $d(Z(R))\subseteq Z(R)$
for any  $d\in \Der R$.

\begin{proposition} \label{P:11}
Let $R$ be a ring and  let $G$ be a group. The
following conditions hold:
\begin{itemize}
\item[$(i)$] if $\delta \in \Der_RR[G]$, then $\delta (g)\in
Z(R)[G]$ for any $g\in G$;
\item[$(ii)$] $\ZDer_RR[G]$ is an ideal of the Lie ring $\Der_R R[G]$;
\item[$(iii)$] if $G$ is a torsion  group such that $\pi (F(R))\cap
\pi (G)=\varnothing$  (respectively $G$ is an abelian  divisible torsion-free group
and $nR=0$ for some  $n\in\mathbb{N}$), then $\ZDer_R R[G]=0$;
\item[$(iv)$] if  $\pi (F(R))\cap \pi (G)=\varnothing$, then $\delta (\tau
G)=0$ for any $\delta \in \ZDer_R R[G]$;
 \item[$(v)$] if $G$ is a
torsion group and $\pi (F(R))\cap \pi (G)=\varnothing$, then
$\Der_RR[G]=0$ if and only if $\delta (G)=0$ for any $\delta \in
\Der R[G]$.
\end{itemize}
\end{proposition}
\begin{proof} $(i)$  Obviously,  because $\delta (g)r=\delta (gr)=\delta (rg)=r\delta (g)$ for any $g\in G$ and $r\in R$.

$(ii)$ If $\delta \in \Der_R R[G]$ and  $\mu \in \ZDer_RR[G]$, then
\[
[\delta ,\mu ](x)=\delta (\mu
(x))-\mu (\delta (x))\in Z(R[G]) \ \mbox{and}\ \ [\delta ,\mu ](r)=0
\]
for any $x\in R[G]$ and $r\in R$.
Hence $[\delta ,\mu ]\in \ZDer_RR[G]$.

$(iii)$ Suppose that  $\delta \in \ZDer_R R[G]$ and $g\in G$.

$a)$ If $G$ is torsion and $n=|g|$, then
\begin{equation} \label{E:3}
0=\delta (1)=\delta (g^n)=ng^{n-1}\delta(g),
\end{equation}
this  gives that
$\delta(g) =0$, so $\delta =0$.

$b)$ Assume that $G$ is abelian  divisible torsion-free and $nR=0$. Then
$nL_\delta (g)=0$ and Eqs. $(\ref{E:2})$ imply  that $\delta
(g^n)=0$ and thus $\delta =0$.

$(iv)$ If  $g\in \tau G$ and  $|g|=n$, then  $\delta (g)=0$ by Eqs. ($\ref{E:3}$).

$(v)$ \  $(\Rightarrow )$ If $\Der_RR[G]=0$, then $\partial_g=0$  for any
$g\in G$ and so
 $G\subseteq Z(R[G])$. Hence $\delta (G)\subseteq Z(R[G])$  for any $\delta \in \Der R[G]$. Since  $g^n=1$
for some $n\in \mathbb{N}$ and $L_\delta :\langle g\rangle \to
R[G]^+$ is a group homomorphism by Lemma  \ref{L:8}$(iii)$, we
have that Eqs. $(\ref{E:2})$ imply that $nL_\delta (g)=0$. Hence $L_\delta (g)=0$, because $\pi (F(R))\cap \pi (G)=\varnothing$. Therefore $\delta (g)=0$, what gives $\delta (G)=0$.

$(\Leftarrow )$ Obviously.
\end{proof}

\begin{proof}[Proof of Theorem $\ref{T:2}$]
Let $\delta \in \Der_RR[G]$. Then $\delta (g)\in Z(R)[G]\subseteq Z(R[G])$ for any $g\in G$ by Proposition \ref{P:11}$(i)$. In as much as $R$ is $|g|$-torsion-free, we conclude that $\delta (g)=0$ be the same argument as in the proof of Proposition \ref{P:10}$(iii)$. Hence $\delta (G)=0$ and consequently $\delta =0$.
\end{proof}

Since $\Der R[G]=\Der_RR[G]$ in a differentially trivial ring $R$,  Theorem \ref{T:2} implies the following.

\begin{corollary} \label{C:16}
Let  $R$ be  a  differentially trivial ring. If  $G$ is a torsion abelian group  such that each prime $p\in \pi (G)$ is invertible in $R$, then $\Der R[G]=0$.
\end{corollary}

\section{Locally inner derivations}

A derivation $\delta \in \Der R[G]$ is called {\it locally inner}
(see \cite{Ikeda_Kawamoto, Kawamoto}) if, for every finite subset $F\subseteq
R[G]$, there exists an inner derivation $\partial_x\in \IDer R[G]$
(depending on $\delta$ and $F$) such that
\[
\delta_{|F}=({\partial_x})_{|F}.
\]
The set  $\LDer R[G]$ of all
locally inner derivations of $R[G]$ is an ideal of the Lie ring
$\Der R[G]$ (see \cite{Ikeda}) and
\[
\IDer R[G]\subseteq \LDer R[G]\subseteq \Der R[G]\quad  \text{and}\quad
\IDer_R R[G]\subseteq \LDer_R R[G]\subseteq \Der_R R[G].
\]

Now as a light extension of \cite[Theorem 2.1]{Ikeda_Kawamoto} we have.

\begin{lemma} \label{L:11}
Let $R$ be a  ring and let $G$ be a locally finite group.
If each prime $p\in \pi (G)$ is invertible in $R$, then all $R$-derivation of $R[G]$ is locally inner.
\end{lemma}

\begin{proof}   The same as the proof of \cite[Theorem 2.1]{Ikeda_Kawamoto}. We prove it here in order to have the
paper more self-contained. Assume that $\delta \in \Der_R R[G]$
and $F$ is a finite subset of $R[G]$. Then there exists a finite
subgroup $H$ of $G$ such that $F\subseteq R[H]$.  Let
\begin{equation} \label{E:4}
x_H=\textstyle\frac{1}{|H|}\sum_{a\in
H}a^{-1}\delta (a)\in R[H].
\end{equation}
If $x:=x_H$,  then for any $y\in F$ and $a\in H$ holds
\[
\begin{split}
\partial_{x}(y)&=\textstyle\frac{1}{|H|}\{ \sum_{a\in H }a^{-1}\delta (a)y-\sum_{a\in H}ya^{-1}\delta (a)\}\\
&=\textstyle\frac{1}{|H|} \{ \sum_{a\in H} a^{-1}\delta (ay)-\sum_{a\in H}a^{-1}a\delta (y)-\sum_{a\in H}ya^{-1}\delta (a)\} \\
&= \textstyle\frac{1}{|H|} \{ \sum_{b\in H}ya^{-1}\delta (b)-|H|\delta (y)-\sum_{a\in H}ya^{-1}\delta (y)\} \\
&=-\delta (y).
\end{split}
\]
Hence $\delta_{|F}={(\partial_{-x})}_{|F}$.  \end{proof}

\smallskip
\noindent
{\bf Example.} Let $S(\mathbb{Q}):=\{f: \mathbb{Q}\to \mathbb{Q}\mid f  \text{  is bijective} \}$  be the group of all permutations of the rational numbers field  $\mathbb{Q}$ and $\sup f:=\{q\in \mathbb{Q}\mid f(q)\not=q\} $ a support of the element $f\in S(\mathbb{Q})$.  The set
$FS(\mathbb{Q}):=\{f\in  S(\mathbb{Q})\mid \sup f  \leq \infty  \}$ is a subgroup of $S(\mathbb{Q})$ and   $R[FS(\mathbb{Q})]$ is a  group ring over a commutative ring $R$. Let $u=\sum_{q\in \mathbb{Q}}f(q)\in (R[FS(\mathbb{Q})])^*$, where $f(q)\in FS(\mathbb{Q})$ and  $\sup f(q)=\{q-1, q\}$.  It follows the following two statements:
\begin{itemize}
\item[$(i)$] \cite[Example]{Burkov} the rule
\[
\delta: R[FS(\mathbb{Q})]\ni x\mapsto [x,u]\in R[FS(\mathbb{Q})]
\]
determines a generalized inner derivation of $R[FS(\mathbb{Q})]$, which is not inner;
\item[$(ii)$] $\delta$ is a locally inner derivation.
\end{itemize}
In fact, if   $\delta$ is an inner  derivation, then $\delta=\partial_a$ for some $a=\sum_{g\in FS(\mathbb{Q})}\alpha_gg\in R[FS(\mathbb{Q})]$. Since this sum  is finite, there exist $\beta,\lambda \in \mathbb{Q}$  and $f\in R[FS(\mathbb{Q})]$ such that
\[
\sup f=\{\beta, \lambda\}\subseteq  \mathop{\cup}_{q\in \mathbb{Q}} \sup f(q) \quad\text{and}\quad  \sup f\cup \big(\mathop{\cup} _{g\in \supp a} \sup g\big) =\varnothing,
\]
where $\supp a$ is the support of $a\in R[G]$. It follows that  $\partial_a(f)=0$ and $\delta(f)\not=0$, a contradiction. Hence $\delta$ is not inner.

If $F$ is a finite subset of $R[FS(\mathbb{Q})]$, then $F\subseteq R[H]$ for a finite subgroup $H$ of $FS(\mathbb{Q})$.
The set  $\supp H=\mathop{\cup}_{a\in H} \supp a$ is finite, so there exists $b\in  R[FS(\mathbb{Q})]$
such that
\[
\supp b\subseteq \supp u\mathop{\cup} \supp H
\]
and $\delta_{\mid H}=(\partial_b)_{\mid H}$. Consequently, $\delta$ is locally  inner on $F$.

Notice that, if $G$ is a non-abelian group, then $gh-hg\neq 0$ for $g\in G$, so $\IDer_R R[G]$ is nonzero.
If $H$ is a subgroup of $G$, then $\mathfrak{I}_R(H)$ is a right ideal of $R[G]$ generated by the set $\{ 1-h\mid h\in H\}$.
If $H$ is normal in $G$, then $\mathfrak{I}_R(H)$ is an ideal of $R[G]$ and there exists some  ring isomorphism
\[
R[G]/\mathfrak{I}_R(H)\cong R[G/H]
\]
 \cite[Proposition 1, p.~17]{Bovdi_book}. Therefore, we have a Lie ring isomorphism
\[
\Der (R[G]/\mathfrak{I}_R(H))\cong \Der R[G/H].
\]
Let $\delta \in \Der R$. An ideal $I$ of $R$ is called {\it a $\delta$-ideal} of $R$  if $\delta (a)\in I$ for all $a\in I$.
If $\delta(a)\in I$ for any  $\delta \in \Der R$, then $I$ is called a $(\Der R)$-{\it ideal}.

\begin{lemma} \label{L:13}
Let $R$ be a ring and let $G$ be a group. If  $H$ is a subgroup of $G$ such that each prime  $p\in \pi (H)$
is invertible in $R$, then the following conditions hold:
\begin{itemize}
\item[$(i)$] if $\delta \in \Der R[G]$ and $H$ is finite, then $\delta (H)=\partial_x(H)$ for some $x\in R[G]$;

\item[$(ii)$] if $H$ is a normal torsion subgroup of $G$, then $\mathfrak{I}_R(H)$ is a $\delta$-ideal for any $\delta \in \Der_RR[G]$.
\end{itemize}
\end{lemma}
\begin{proof} $(i)$ In fact, from proof of Lemma \ref{L:11} it holds that $x:=x_H$ is of the form $(\ref{E:4})$.

$(ii)$ If  $h\in t(H)$, then $\delta (h)=\partial_x (h)$ for some $x=\sum_{t\in G}\alpha_tt\in Z(R)[G]$ and
\[
\delta (rg(h-1))=r\delta (g)(h-1)+rg\delta (h)\qquad (r\in R, g\in G, h\in H).
\]
Moreover, \quad $rg\delta  (h)=\sum_{t\in G}r\alpha_tg(th-ht)=\sum_{t\in G}r\alpha_tgt(1-h^{-1}t^{-1}ht)\in \mathfrak{I}_R(H)$. \quad It  follows that    $\delta (rg(h-1))\in \mathfrak{I}_R(H)$ and $\mathfrak{I}_R(H)$ is a $\delta$-ideal.
\end{proof}

If $R$ is a differentially simple ring (i.e., $R^2\neq 0$ and $R$ has no proper $(\Der R)$-ideals) with a minimal two-sided ideal, then (see \cite[Main Theorem]{Block}) the ring  $R$ is simple  or there exist   $n\in\mathbb{N}$ and a simple ring $S$ of characteristic $p>0$ such that $R=S[G]$, where $G$ is a  direct sum of $n$ copies of a cyclic group of order $p$. Therefore,   Lemma \ref{L:13}$(ii)$ is not true in the modular case.

\begin{corollary} \label{C:14}
Let $R$ be a  ring. If  $G$ is a finite group such that each prime $p\in \pi (G)$ is  invertible in $R$, then  every $R$-derivation of $R[G]$  is inner.
\end{corollary}

\begin{proof}
Let $H$ be  a finite subgroup of $G$,  $b\in H$ and $\delta \in \Der_RR[G]$. It is easy to see that  $\delta (b)=-\partial_x(b)$ by Lemma \ref{L:11}, where $x$ is the same as in Eq. $(\ref{E:4})$.
\end{proof}

This corollary earlier was proved for a finite group $G$ and the following rings $R$:
\begin{itemize}
\item[$\bullet$] $R$ is a semiprime ring with some restrictions  (see \cite[Theorem 1.1]{Ferrero_Giambruno_Milies});
 \item[$\bullet$]  $R$ is a field of characteristic zero (see \cite[Corollary
2.2]{Ikeda_Kawamoto});

 \item[$\bullet$]  $R$ is the integer numbers ring (see \cite[Theorem 1]{Spiegel});

\item[$\bullet$]  $R$ is  commutative with some restrictions  (see  \cite[Corollary]{Burkov},   \cite[p.~490]{Curtis_Reiner}
and \cite[p.~76]{DeMeyer_Ingraham}).
\end{itemize}

\smallskip
\noindent
{\bf Example.} Let $R$ be a commutative ring,
$G$ a torsion  abelian group and  $\delta \in \Der R[G]$.

$(a)$ If $R$ is $n$-torsion-free and $\exp(G)=n\in\mathbb{N}$
(respectively $R^+$ is torsion-free),
then $\Der_RR[G]=0$ in view of Proposition $\ref{P:10}(iii)$.

$(b)$ If $R$ is a   domain of characteristic $0$ which fraction
field $Q(R)$ is an algebraic extension of the rational numbers field
$\mathbb Q$, then $\Der_R  R[G]=0$.

Indeed, $Q(R)[G]$ is an algebraic extension over $Q(R)$ (see e.g.
\cite[Theorem 28.1]{Pas71}) and  we conclude that for every element $g\in G$
there exists its minimal polynomial
\[
m_g=X^n+a_1X^{n-1}+\cdots +a_{n-1}X+a_n\in Q(R)[X].
\]
 Then
$bm_g\in R[X]$ for some $0\neq b\in R$. If $\delta \in
\Der_RR[G]$, then
\[
0=\delta (0)=\delta
(bm_g(g))=b(ng^{n-1}+(n-1)a_1g^{n-2}+\cdots +a_{n-1})\delta (g)
\]
what forces that $\delta (g)=0$.

$(c)$ If $R$ is a domain of characteristic $p>0$, $R=\{x^p\mid
x\in R\}$ and $p\notin \pi (G)$, then $\Der R[G]=0$.

Clearly,  for every $r\in R$ there exists $x\in R$ such that
$r=x^p$ and then  $\delta (r)=px^{p-1}\delta (x)=0$. Moreover,   $\delta (G)=0$ by Eqs. $(\ref{E:3})$  and the assertion follows.

\begin{proof}[Proof of Theorem $\ref{T:1}$]
Let $G=D\cdot F$ be  a Chernikov group, where $D$ is a
normal,  countable,  divisible abelian group and $F$ is a finite group. Obviously,   $D$   has an ascending series $\{
D_n\}_{ n \in {\mathbb N}}$ of finite $G$-normal subgroups,  such that $D=\bigcup_{n=1}^\infty
D_n$.   Let  $\delta \in \Der_RR[G]$. There exist elements $x_1:=x_{D_1}$ and
$x_n:=x_{D_n}$ of the form $(\ref{E:4})$ (see  Lemma \ref{L:11}) such that
\[
\delta
(D_1)={\partial_{-x_1}}(D_1)  \qquad  \text{and}\qquad  \delta(D_n)={\partial_{-x_n}}(D_n).
\]
Evidently, ${\partial_{-x_s}}(D_s)=\delta (D_s)={\partial_{-x_n}}(D_s)$ for all   $s\leq n\in \mathbb{Z}$, so
$x_s-x_n\in C_{R[G]}(D_s)$, in which  $C_{R[G]}(D_s)$  is the centralizer of $D_s$ in $R[G]$.
Then we have the following descending chain
\[
C_{R[G]}(D_1)\geq \cdots \geq C_{R[G]}(D_s)\geq \cdots .
\]
Since $R[D]$ is a subgroup of finite index in the additive group of $R[G]$ and $R[D]\subseteq C_{R[G]}(D_s)$,   there exists  $m\in\mathbb{N}$ such that
\[
 C_{R[G]}(D_m)=C_{R[G]}(D_{m+1})=\cdots .
 \]
It follows that  $x_q-x_m\in C_{R[G]}(D_m)$ for all  $q\geq
m$. Hence $\delta_{|D} =(\partial_{x_m})_{|D}$ is inner. Moreover,
$D_mF$ is finite and $\delta_{|D_mF} =(\partial_{a})_{|D_mF}$ for
some $a\in R[G]$. However  $x_m-a\in C_{R[G]}(D_m)$ and so
$\partial_{x_m}=\partial_{a}$ on $D$. Consequently,  $\delta
=\partial_{a}$. \end{proof}

\section{Inner $R$-derivations of $R[G]$}

If $R$ is  commutative, then $\IDer R[G]=\IDer_RR[G]$. Additionally, if  $G$ is an abelian group, then $\IDer_RR[G]=0$. Clearly, the  center $Z(B)$ of $B$ is an ideal of the Lie ring $B^L$.

\begin{lemma} \label{GGC4}
Let $R$ be a ring and let $G$ be a group. There exists a Lie ring isomorphism
\[(Z(R)[G])^L/Z(Z(R)[G])\cong \IDer_RR[G].\]
\end{lemma}
\begin{proof} The rule $\varphi :(Z(R)[G])^L \ni \alpha \mapsto \partial_\alpha \in \IDer_RR[G]$
satisfies
\[
\varphi ([\alpha ,\beta ])=\partial_{[\alpha ,\beta
]}=[\partial_\alpha ,\partial_\beta ]= [\varphi (\alpha ),\varphi
(\beta )]
\]
i.e.,  $\varphi$ is a Lie homomorphism. Finally,
$\varphi (\alpha ) =0$ iff $ \partial_\alpha
=0$ iff  $ \alpha \in Z(Z(R)[G])$.
\end{proof}

\begin{corollary} Let $R$ be a ring and let $G$ be a group. The ring  $\IDer_RR[G]$ is an abelian Lie ring if and only if the derived subgroup $G'$ is central.
\end{corollary}
\begin{proof} Since $[\partial_x,\partial_y]=\partial_{[x,y]}$ for any $x,y\in R[G]$, we conclude that $[\partial_x,\partial_y]=0$ iff $[x,y]\in Z(R[G])$ what implies that $G'\subseteq Z(G)$.
\end{proof}

Since no each derivation of $R[G]$ is inner in the case of a
locally finite group $G$  (see  \cite[Example]{Burkov}), we
obtain the next.
\begin{lemma} \label{L:18} Let $R$ be a ring,  let $G$ be a locally finite group
such that each $p\in \pi (G)$ is invertible in $R$ and $\delta \in
\Der_RR[G]$. If the set   $\delta (G)$  is finite, then $\delta \in
\IDer_RR[G]$.
\end{lemma}
\begin{proof}
The set $Q:=\{ g\in G\mid \delta (g)=0\}$ is a subgroup of $G$
because
\[
0=\delta (1)=\delta (hh^{-1})=\delta (h)h^{-1}+h\delta (h^{-1})=h\delta (h^{-1}),
\]
so $\delta (h^{-1})=0$  and
$\delta (gh^{-1})=\delta (g)h^{-1}+g\delta (h^{-1})=0$ for any
$g,h\in Q$.
Clearly,   $\delta (G)=\{ \delta (g_i)\mid i=1,\ldots, n\}$
and the subgroup $H=\langle g_1,\ldots ,g_n\rangle \subseteq G$ is finite (where $n=|\delta (G)|$), so we
conclude  that
$\delta_{|H}={(\partial_x)}_{|H}$
 for some $x\in Z(R)[G]$ by Lemma \ref{L:13}$(i)$. \end{proof}

\begin{lemma}  \label{L:19}
Let $R$ be a ring and  let $G$ be a  group
such that each prime $p\in \pi (G)$ is invertible in $R$. If   $\delta
\in \Der R[G]$ and $g\in \tau G$, then the following conditions hold:
\begin{itemize}
\item[$(i)$] \cite[Lemma 1]{Burkov} if $t\in G$ and $\delta
(g)=\sum_{h\in G}\alpha_{g,h}h\in R[G]$,  then $gt=tg$ implies
that $\alpha_{g,t}=0$;
\item[$(ii)$] if $G$ is torsion, then
$\delta (Z(G))=0$ for any $\delta \in \Der R[G]$.
\end{itemize}
\end{lemma}
\begin{proof} $(i)$ If $n=|g|$, then
\[
\begin{split}
0=\delta (1)&=\delta (g^n)=\sum_{i+j=n-1}g^i\delta(g)g^j=\delta (g^n)=\\
&=\sum_{i+j=n-1}g^i(\sum_{h\in
G}\alpha_{g,h}h)g^j=\sum_{\scriptsize h\in G,\
i+j=n-1}\alpha_{g,h}g^ihg^j\\
&=\sum_{\scriptsize i+j=n-1}\alpha_{g,h}g^itg^j+\sum_{\scriptsize h\in G\setminus
\{t\},\ i+j=n-1}\alpha_{g,h}g^ihg^j\\
&=n\alpha_{g,t}t+\sum_{\scriptsize h\in G\setminus \{t\},\
i+j=n-1}\alpha_{g,h}g^ihg^j
\end{split}
\]
what gives that $\alpha_{g,t}=0$. \quad The part $(ii)$ it holds from $(i)$.
\end{proof}

We obtain the next generalization of \cite[Theorem 1.1]{Ferrero_Giambruno_Milies}.

\begin{corollary} \label{C:20}
Let $R$ be a ring, $G$ a torsion $FC$-group such that each $p\in \pi (G)$ is invertible in $R$ and
$\delta \in \Der_RR[G]$. The following conditions hold: \begin{itemize}
\item[$(i)$] $\delta$ is inner if and only if the set $\delta (G)$
is finite;
\item[$(ii)$] if $G$ is centre-by-finite, then every
$R$-derivation of $R[G]$ is inner.
\end{itemize}
\end{corollary}
\begin{proof} Let $\delta \in \Der_RR[G]$.

$(i)$ $(\Leftarrow )$ It follows by  Lemma \ref{L:18}.

$(\Rightarrow )$ If $\delta$ is inner, then  there exists $x\in
Z(R)[G]$ such that  $\delta=\partial_x$. Since the index $|G:C_G(x)|$ is finite,   the
image $\partial_x(G)$ is finite.

$(ii)$ In as much as $|G:Z(G)|<\infty$, we deduce that $\delta (G)$ is finite in view of Lemma \ref{L:19}$(ii)$. The rest holds from the part $(i)$.      \end{proof}

\begin{proof}[ Proof of the Theorem $\ref{T:3}$]
Let $G$ be an  infinite group with an ascending series of its subgroups
\[
H\leq
\langle H,x_1\rangle \leq \cdots \leq \langle H,x_1,\ldots ,
x_n\rangle \leq  \cdots \] such that $G=\bigcup_{n=1}^\infty
\langle H,x_1,\ldots , x_n\rangle$. There exist  $x,y_n\in R[G]$  by Lemma \ref{L:13}$(i)$, such that
\[
\delta_{|H}=({\partial_x})_{|H}\qquad \text{and}\qquad  \delta_{|\langle
H,x_1,\ldots , x_n\rangle}=({\partial_{y_n}})_{|\langle H,x_1,\ldots, x_n\rangle}.
\]
 Then $y_n-x\in C_{R[G]}(H)$ and so there exists
$n_0\in \mathbb{N}$ such that $y_m=y_{n_0}$ for all $m\geq n_0$. This
yields that $\delta =\partial_{y_m}$.
\end{proof}

\begin{proof} [Proof of the Theorem $\ref{T:4}$] The group ring $R[G]$ is prime and every its non-trivial idempotent is non-central.

$(i)$ Let $G$ be non-abelian. The equation  $(\ref{EE:5})$ and \cite[Corollary 6]{Lanski} imply that either $[G,L]=0$   or $[U(R),M]=0$, where $L$ is some non-central Lie ideal and $M$ is some non-central ideal of $R$, which is impossible by \cite[Lemma 2]{Bergen_Herstein_Kerr}.

$(ii)$ The group   $V(R[G])$ is central  by \cite[Corollary 6]{Lanski} and \cite[Lemma 2]{Bergen_Herstein_Kerr} and so
$\delta (\tau G)=0$ for $\delta \in \Der R[G]$ by Lemma
\ref{L:19}.
\end{proof}

\section{Nilpotency and solvability of derivation rings}

A group $G$ is called {\it $p$-abelian}  ($p>0$) if its
commutator subgroup $G'$ is a finite $p$-group.  A ring $R$ is called {\it Lie hypercentral}, if for each sequence $x,x_1,\ldots ,x_n, \ldots \in R$,  there exists $m\in\mathbb{N}$ such that $[x,x_1,\ldots ,x_m]=0$. Analogously  we can defined the notion of a {\it hypercentral} Lie ring.

\begin{proposition} \label{P:21}
Let $R$ be a ring and let $G$ be a group.
The following conditions hold:
\begin{itemize}
\item[$(i)$] if $\IDer_RR[G]$ is a nilpotent (respectively
solvable) Lie ring, then the unit group $U(Z(R)[G])$ (and
consequently $G$) is a nilpotent (respectively solvable) group;
\item[$(ii)$] if\ $R$ is a division ring of characteristic
$p\geq 0$, then:
\begin{itemize}
\item[$(a)$] the Lie ring $\IDer_RR[G]$ is nilpotent if and only if $G$ is $p$-abelian and nilpotent;
\item[$(b)$] for $p\neq 2$, $\IDer_RR[G]$ is solvable if and only if  $G$ is $p$-abelian;
\item[$(c)$] for $p=2$, $\IDer_RR[G]$ is solvable if and only if  $G$ has  a $2$-abelian subgroup of index at most $2$;
\end{itemize}
\item[$(iii)$] the Lie ring $\IDer_RR[G]$ is hypercentral if and only if one of the following conditions holds:
\begin{itemize}
\item[$(d)$]  $G$ is abelian;
\item[$(e)$] $R$ is of characteristic $p^m$ and $G$ is a nilpotent $p$-abelian group.
\end{itemize}
\end{itemize}
\end{proposition}
\begin{proof}  $(i)$ Assume that $\IDer_RR[G]$ is nilpotent
(respectively solvable). The Lie ring $(Z(R)[G])^L$ is nilpotent (respectively solvable) by Lemma \ref{GGC4}. Then there exists  $n\in\mathbb{N}$ such that
\[
\gamma_nG\leq \gamma_nU(Z(R)[G])\leq [\underbrace{Z(R)[G],\ldots ,Z(R)[G]}_{\scriptsize n\ \text{times}}]+1=1
\]
by \cite[Theorem A]{Gupta_Levin} (respectively
\[
 G^{(n)}-1\leq U(Z(R)[G])^{(n)}-1\leq
((Z(R)[G])^L)^{(n)}=0
\]
in view of \cite[Lemma  1.2]{Sharma_Srivastava84}).

$(ii)$ $(\Rightarrow )$ If $\IDer_RR[G]$ is nilpotent (respectively solvable), then $(Z(R)[G])^L$ is nilpotent (respectively solvable) by Lemma \ref{GGC4}. The rest follows from \cite[Theorem]{PPS73}.

 $(\Leftarrow )$  In as much as $Z(R)[G]$ is Lie nilpotent (respectively Lie solvable) by \cite[Theorem]{PPS73}, we deduce that $\IDer_RR[G]$ is nilpotent (respectively  solvable) by Lemma \ref{GGC4}.

 $(iii)$ If follows in view of \cite[Theorem]{Bovdi_Khripta86} and Lemma \ref{GGC4}.
\end{proof}

\begin{lemma} \label{P:22}
Let $P$ be an ideal in a ring $B$.
The following conditions hold:
\begin{itemize}
\item[$(i)$] if $\delta (P)\subseteq P$ for some $\delta \in \Der
B$, then
\[
\overline{\delta}:B/P\ni a+P\mapsto \delta (a)+P\in
B/P\] is a derivation of the quotient ring $B/P$;
\item[$(ii)$] if
$A$ is a nilpotent (respectively solvable) subring of $\Der B$ and
$\delta (P)\subseteq P$ for any $\delta \in A$, then
$\overline{A}:=\{ \overline{\delta}\mid \delta \in A\}$ is a
nilpotent (respectively solvable) subring of $\Der (B/P)$.
\end{itemize}
\end{lemma}
\begin{proof} Evidently.
\end{proof}

\begin{lemma} \label{L:23}
Let $B$ be a ring. The  following conditions hold:
\begin{itemize}
\item[$(i)$] if $\delta \in Z(\Der B)$, then $\delta (R)\subseteq
Z(B)$;
 \item[$(ii)$] if $\delta \in Z(\Der B)$, then $\delta
(Z(B))\subseteq \ann D:=\{ r\in B\mid r(\Der B)=0\}$;
\item[$(iii)$] if $B$ is commutative and
$\delta \in Z(\Der B)$, then:
\begin{itemize}
\item[$(a)$] if $\delta$ is surjective as a map, then $B^2=0$;
\item[$(b)$] if $\ann B=0$, then $\delta =0$;
\end{itemize}
\item[$(iv)$] if  $\IDer B$ is Lie nilpotent (or
 equivalently $B$ is Lie nilpotent), then $C(B)\subseteq {\mathbb
 P}(B)$;
\item[$(v)$] if $\IDer B$ is Lie solvable (or
equivalently $B$ is Lie solvable), then $[B^{(n)},B]\subseteq
{\mathbb
 P}(B)$ for some integer $n\geq 0$;
 \item[$(vi)$] if $B$ is semiprime with the solvable (in particular, nilpotent) Lie ring   $\IDer B$, then
$B$ is commutative.
\end{itemize}
\end{lemma}
\begin{proof} Assume that $\delta \in Z(\Der B)$.

$(i)$ In fact, $\partial_x\delta (y)=\delta \partial_x(y)$ for any
$x,y\in B$ and so $[x,\delta (y)]=[\delta (x),y]+[x,\delta (y)]$
what gives that $[\delta (x),y]=0$.

$(ii)$ If $a\in Z(B)$ and $d\in \Der B$, then $ad\in \Der B$. Since
\begin{equation}\label{E:6}
(ad)\delta =\delta (ad)=\delta (a)d+a(\delta d),
\end{equation} we deduce that $\delta (a)d=0$.

$(iii)$ It holds from the part $(ii)$.

$(iv)$ Let $P$ be a prime ideal of $B$. Since $\IDer (B/P)$ is
nilpotent, $B_1:=B/P$ is Lie nilpotent. Clearly, $B_1[[B_1,B_1],B_1]B_1$ is a nilpotent ideal of $B_1$ by
\cite{jennings} (this is also was proved in \cite[Lemma 2.2]{Amberg_Sysak} and \cite[Corollary 2.4]{Szigeti_Van_Wyk}).  It is easy to see that  $\ann B_1=0$, so
$[B_1,B_1]\subseteq Z(B_1)$.  The commutator ideal $C(B_1)$ is nil by \cite [Lemma 1.7]{Bell_Klein}. Consequently,  $B_1$ has a nilpotent ideal
contained in $C(B_1)$ and so $C(B_1)=0$. This means that $B/P$ is
commutative and thus $C(B)\subseteq {\mathbb P}(R)$.

$(v)$  If $\IDer B$ is solvable of length $n$, then $B^L$ is
solvable of length $\leq n+1$. Let $P$ be a prime ideal of $B$. Obviously,
$(B/P)^{(n)}$ is a commutative Lie ideal of the prime ring $B/P$.
If $\charak B/P\neq 2$, then  $(B/P)^{(n)}\subseteq Z(B/P)$ by
\cite[Lemma 2]{Bergen_Herstein_Kerr}.

Now let  $\charak B/P=2$. If $Z(B/P)=0$, then $(B/P)^{(n)}$ is nil of the index $2$ by
\cite[Lemma 2]{Lanski_Montgomery} and so $(B/P)^{(n)}=0$ by \cite[Lemma
1]{Lanski_Montgomery}. Hence $B^{(n)}\subseteq P$. Therefore, we assume that
$Z(B/P)\neq 0$. Consequently,   $(B/P)Z(B/P)^{-1}$ is  simple $4$-dimensional
over its center by \cite[Theorem 4]{Lanski_Montgomery}, in which  $(B/P)Z(B/P)^{-1}$ is the localization of $B/P$ at $Z(B/P)$, and so $(B/P)^{(n)}$ is
central by \cite[Lemma 6]{Lanski_Montgomery}.

In both cases $[B^{(n)},B]\subseteq P$, so $[B^{(n)},B]\subseteq {\mathbb P}(B)$.

$(vi)$ It holds in view of the part $(v)$.
\end{proof}

\begin{proof} [Proof of the Theorem $\ref{T:5}$]  Let $D:=\Der B$.

$(i)$ Assume that $D$ is nilpotent. The ring  $B$
is commutative  by Lemma \ref{L:23}$(iv)$ and  $\delta (B)^2=0$
for any $\delta \in Z(D)$ by Lemma \ref{L:23}$(ii)$ and   $\delta
(B)=0$ in view of the semiprimeness of $B$ and Lemma \ref{L:23}$(ii)$. This means that $\delta =0$ and thus $D=0$.

$(ii)$ Now assume that the Lie ring $D$ is solvable. The ring $B$ is a commutative by Lemma \ref{L:23}$(iv)$ and consequently $D$ is a left $B$-algebra. Assume that the algebra $D$ is nonzero, $A$ is the last nonzero member of its derived chain, $d,\delta \in
A$ and $b\in B$. In as much as $A$ is abelian and $bd\in A$, Eqs. $(\ref{E:6})$ imply that $\delta (b)d=0$. This yields   that $(\delta(B)B)^2=0$. The semiprimeness of $B$ implies  $\delta (B)=0$, a contradiction.
\end{proof}

\begin{proof}[Proof of the Theorem $\ref{T:6}$]
 If every prime $p\in \pi (G)$ is invertible in $\mathbb F$,
then ${\mathbb F}[G]$ is a semiprime ring (see \cite{Connell}). The rest holds from  Theorem  \ref{T:5} and
Proposition \ref{P:21}.
\end{proof}

We need the next
\begin{lemma} \cite[Lemma 1]{Abu-Khuzam} \label{L:25}
Let  $B$ be a ring and let $x,y\in B$. If  $[[x,y],x]=0$,   then
\[
[x^k,y]=kx^{k-1}[x,y]\qquad  (k\in\mathbb{N}).
\]
\end{lemma}
A  group $G$ is called {\it $n$-divisible} ($n\in\mathbb{N}$) if, for each $g\in
G$, there exists $h\in G$ such that  $g=h^n$.

\begin{proposition}\label{L:26}
Let $R$ be a ring and let $G$ be a torsion group such that each $p\in \pi(G)$ is invertible in $R$ (respectively $nR=0$ and $G$ a $n$-divisible group for some $n\in \mathbb{N}$). If $G$ is an Engel  set in $R[G]$, i.e.
for each $g,h\in G$ there exists $m=m(g,h)\in\mathbb{N}$
such that $[g,_m h]=0$,  then
$G$ is abelian and $\Der_RR[G]=0$.
\end{proposition}
\begin{proof}
Let $g,h\in G$. Since $G$ is an Engel  set, there exists $m\in\mathbb{N}$
such that $[g,_m h]:=[a,h,h]=0$, where $a:=\begin{cases} [g,_{m-2} h] & \text{  if   } m\geq 3;\\
g& \text{  if   }  m=2.
\end{cases}
$ \quad
Now, if each $p\in \pi(G)$ is invertible in $R$, then
\[
0=[h^{|h|},a]=|h|h^{|h|-1}[h,a]
\]
(see  Lemma \ref{L:25}).  If $nR=0$ and $G$ is a $n$-divisible group, then $[h^n,a]=n\big(h^{n-1}[h,a]\big)=0$ (see  Lemma \ref{L:25}). Thus  $[g,a]=0$ in both cases, so $[g,h]=0$ by induction.  Consequently,   $G$ is abelian and  $\Der_RR[G]=0$ in view of Theorem \ref{T:2}.
\end{proof}

\section{Case of  nilpotent groups}

Each nilpotent group $G$ of class  $m$ has a  central series
\[
1=Z_0<Z_1<\cdots <Z_{m-1}<Z_m=G
\]
in which $Z(G/Z_{i-1})=Z_i/Z_{i-1}$ for all $i=1,\ldots, m$.

\begin{lemma}\label{L:27}
Let $R$ be a ring such that $nR=0$ for some  $n\in \mathbb{N}$, $G$ a divisible torsion-free nilpotent group of the nilpotent length $m$  and $\delta \in \Der R[G]$. Then $\mathfrak{I}_R(Z_k)$ is a $\delta$-ideal for each $k=1,\ldots ,m$.
\end{lemma}
\begin{proof} Each subgroup $Z_k$ is  divisible and normal in $G$. Moreover,  $\delta (Z_1)\subseteq Z(R[G])$. Now we have
\[
\delta (rg(z-1))=\delta (r)g(h-1)+r\delta (g)(h-1)+rg\delta (h)\quad \text{and}\quad  \delta (h^n)=nh^{n-1}\delta (h)= 0
\]
for any $r\in R$, $g\in G$ and $z\in Z_1$.  We deduce that $\delta (\mathfrak{I}_R(Z_1))\subseteq  \mathfrak{I}_R(Z_1)$.
 There is a Lie ring isomorphism
\[
\Der R[G/Z_1]\ni \overline{\delta}_\varphi \mapsto \overline{\delta}\in \Der R[G]/\mathfrak{I}_R(Z_1)
\]
such that
\[
\overline{\delta}:R[G]/\mathfrak{I}_R(Z_1)\ni a+\mathfrak{I}_R(Z_1)\mapsto \delta (a)+\mathfrak{I}_R(Z_1)\in R[G]/\mathfrak{I}_R(Z_1).
\]
In as much as
\[
\overline{\delta}_\varphi(\mathfrak{I}_R(Z_2/Z_1))\subseteq
\mathfrak{I}_R(Z_2/Z_1)
\]
 and we have the following ring isomorphisms
\[
R[G/Z_1]/\mathfrak{I}_R(Z_2/Z_1)\cong R[(G/Z_1)/(Z_2/Z_1)]\cong
R[G/Z_2]\cong R[G]/\mathfrak{I}_R(Z_2),
\]
 we conclude that $\delta
(\mathfrak{I}_R(Z_2))\subseteq \mathfrak{I}_R(Z_2)$. Thus the assertion  follows by
induction.
\end{proof}

\begin{proposition} \label{P:28}
Let $R[G]$ be the group ring of a nilpotent group $G$ of class  $cl(G)=m\geq 2$ over a
ring $R$. If $G$ is  torsion and $\pi (G)\cap \pi (F(R))=\varnothing$ (respectively $G$ is torsion-free and $nR=0$ for some $n\in\mathbb{N}$), then each $\mathfrak{I}_R(Z_i)$ is a $\delta$-ideal ($i\geq 1$) and
\[
\delta_1\delta_2\cdots \delta_{m-1}(R[G])\subseteq \mathfrak{I}_R(Z_1) \qquad\quad  (\delta ,\delta_1, \ldots, \delta_{m-1}\in  \Der_RR[G]).
\]
Moreover, if  $G$ is a torsion abelian group and $\pi (G)\cap \pi (F(R))=\varnothing$ (respectively $G$ is an abelian torsion-free group and $nR=0$ for some $n\in\mathbb{N}$), then $\Der_RR[G]=0$.
\end{proposition}
\begin{proof}
Assume also that $\delta \in \Der_RR[G]$, $r\in R$, $g\in G$, $a\in Z_1$  and
$\delta (g)=\sum_{f\in G}x_ff\in Z(R)[G]$.
Obviously,  $\delta  (Z_1)\subseteq \delta (Z(R[G]))\subseteq Z(R[G])$ and $L_\delta :Z_1\to R^+$ is a group homomorphism by Lemma \ref{L:8}$(iii)$. Therefore, $a^{-1}\delta (a^n)=\delta (1)=0$ if $a$ is of order $n$  (respectively $nL_\delta (a)=0$) and thus $\delta (a)=0$ in view of Eqs. $(\ref{E:2})$. So, we obtain that
$\delta (Z_1)=0$. The ideal  $\mathfrak{I}_R(Z_1)$ is a $\delta$-ideal by Lemma \ref{L:13}.
This implies that $\mathfrak{I}_R(Z_1)$ is a $\delta$-ideal and so
\[
\overline{\delta}:R[G]/\mathfrak{I}_R(Z_1)\ni a+\mathfrak{I}_R(Z_1)\mapsto \delta (a)+\mathfrak{I}_R(Z_1)\in R[G]/\mathfrak{I}_R(Z_1)
\]
is a derivation of the quotient ring $R[G]/\mathfrak{I}_R(Z_1)$. A ring isomorphism
$
\varphi :R[G]/\mathfrak{I}_R(Z_1) \to R[G/Z_1]$ induces a Lie ring isomorphism
\[
\Der (R[G]/\mathfrak{I}_R(Z_1)\ni \overline{\delta}\mapsto \overline{\delta}_\varphi \in
 \Der R[G/Z_1]\] such that $\overline{\delta}_\varphi(\overline{Z}_2)=\overline{0}$, where $\overline{Z}_2:=Z_2/Z_1$, what forces that $\delta (Z_2)\subseteq \mathfrak{I}_R(Z_1)$. By induction, we deduce that  $\delta (Z_i)\subseteq \mathfrak{I}_R(Z_{i-1})$ and so
\[
[Z_i,G]\subseteq \mathfrak{I}_R(Z_{i-1})\qquad (i=2,\ldots ,m).
 \]
Moreover, each $\mathfrak{I}_R(Z_{i})$ is a $\delta$-ideal  by Lemma \ref{L:13}.

Now, since $\Der_RR[G/Z_{m-1}]=0$, we conclude  that $\delta (R[G])\subseteq \mathfrak{I}_R(Z_{m-1})\subseteq R[Z_{m-1}]$.
By induction, we obtain that $\delta (R[Z_i])\subseteq \mathfrak{I}_R(Z_{i-1})$ for any $i=2,\ldots ,m$ and the result follows.
\end{proof}

\begin{proof}[Proof of Theorem $\ref{T:7}$] Let $R$ be a ring,
$G=\{ x_n\mid n\in \mathbb{Z}\}$  a countable torsion-free abelian group. Let    $\alpha
=\sum_{n\in \mathbb{Z}}a_nx_n\in R[G]$, $ z=\sum_{n\in \mathbb{Z}}z_nx_n \in Z(R[G])$, $\delta \in \Der R[G]$, $x\in G$ and  $r\in R$. Clearly,   $Z(R[G])=Z(R)[G]$.  Since $xr=rx$,
\[
\delta (x)r+x\delta (r)=\delta (r)x+r\delta (x)
\] and  $\delta (x)r=r\delta (x)$. Hence $\delta
(x)\in Z(R[G])=Z(R)[G]$. In addition,
\[
0=\delta (1)=\delta
(xx^{-1})=\delta (x)x^{-1}+x\delta (x^{-1})\] and so
$\delta
(x^{-1})=-x^{-1}\delta (x)x^{-1}=-x^{-2}\delta (x)$. Thus $\delta (\alpha )=\sum_{n\in \mathbb{Z}}\delta
(a_n)x_n+\sum_{n\in \mathbb{Z}}a_n\delta (x_n)$.

If  $a,b\in R\subset R[G]$,  then $\delta (a), \delta (b)\in R[G]$, so we can write
\[
\delta (a)=\sum_{n\in \mathbb{Z}}D_n(a)x_n\quad \text{and} \quad  \delta (b)=\sum_{n\in \mathbb{Z}} D_n(b)x_n, \qquad (D_n(a), D_n(b)\in R)
\]
in which  almost all coefficients  $D_n(a)$ and $D_n(b)$ are zero. Now
\[
\sum_{n\in \mathbb{Z}} (D_n(a)+D_n(b))x_n=\delta (a)+\delta (b)=\delta (a+b)=
\sum_{n\in \mathbb{Z}} D_n(a+b)x_n
\]
implies that
$D_n(a+b)=D_n(a)+D_n(b)$,
\[
\begin{split} \sum_{n\in \mathbb{Z}} D_n(ab)x_n&=\delta (ab)= \delta (a)b+a\delta (b)=\\
&=(\sum_{n\in \mathbb{Z}} D_n(a)x_n)b+ a(\sum_{n\in \mathbb{Z}} D_n(b)x_n)=\\
 &=\sum_{n\in \mathbb{Z}} (D_n(a)b+aD_n(b))x_n\end{split}
\]
implies that  $D_n(a b)=D_n(a)b+aD_n(b)$.

Thus $D_n\in \Der R$ for any $n\in \mathbb{Z}$ and we have  the following Lie ring
isomorphism
\[
\Der R[G] \ni \delta\mapsto ({\{
D_n\}_{n\in \mathbb{Z}}}, \{ \delta (x_n)\}_{n\in \mathbb{Z}}
)\in {\mathcal LF}(R)\oplus (Z(R[G]))^{\mathbb Z}.
\]
Consequently,  $\Der_R R[G]\cong (Z(R)[G])^{\mathbb Z}$.  Finally,
$\partial_{\alpha}(x)=0$ and so $\IDer_R R[G] =0$,  i.e., each
nonzero $R$-derivation of $R[G]$ is outer. \end{proof}

For example, if ${\mathbb Q}^+$ is the additive  group of the rational numbers field ${\mathbb Q}$, then $\Der_RR[{\mathbb Q}^+]=0$.

\begin{corollary}\label{C:29}
Let $R$ be a ring. If  $G$ is a countable abelian  group such that each prime $p\in \pi (G)$ is invertible in $R$, then $\Der_RR[G]=0$ or each nonzero $R$-derivation of $R[G]$ is inner.
\end{corollary}

\section{Acknowledgment}
We would like to express our deep gratitude to the referee for the thoughtful and constructive review of our manuscript.

\end{document}